\documentclass[a4paper,10pt]{article}
\usepackage[utf8]{inputenc}
\usepackage{amsmath,pstricks,amsthm,amssymb,ifthen,color}
\usepackage{lineno,ifthen}
\usepackage{pstricks,pst-sigsys}

\newtheorem{thm}{Theorem}
\newtheorem{lem}{Lemma}
\newtheorem{cor}{Corollary}
\theoremstyle{definition}
\newtheorem{definition}{Definition}
\newtheorem{example}{Example}

\def \arxiv {1}

\newcommand{\pmf}[1]{p_{#1}}

\newcommand{\dom}[1]{\mathcal{#1}}

\newcommand{\Prob}{\mathbb{P}}

\newcommand{\card}[1]{\mathrm{card}(#1)}

\newcommand{\Xvec}{\mathbf{X}}
\newcommand{\Yvec}{\mathbf{Y}}
\newcommand{\Zvec}{\mathbf{Z}}
\newcommand{\Wvec}{\mathbf{W}}

\newcommand{\limn}{\lim_{n\to\infty}}

\newcommand{\Mar}[1]{\mathsf{HMC}(#1)}
\newcommand{\Pm}{\mathbf{P}}
\newcommand{\Pment}[2]{P_{#1\to #2}}
\newcommand{\Qm}{\mathbf{Q}}
\newcommand{\Qment}[2]{Q_{#1\to #2}}
\newcommand{\Mvec}{\mathbf{M}}

\newcommand{\pmfy}[2][\empty]{\ifthenelse{\equal{#1}{\empty}}{\pmf{Y#2}(y#2)}{\pmf{Y#2|Y#1}(y#2|y#1)}}
\newcommand{\pmfx}[2][\empty]{\ifthenelse{\equal{#1}{\empty}}{\pmf{X#2}(x#2)}{\pmf{X#2|X#1}(x#2|x#1)}}
\newcommand{\pmfm}[2][\empty]{\ifthenelse{\equal{#1}{\empty}}{\pmf{Z#2}(y#2)}{\pmf{Z#2|Z#1}(y#2|y#1)}}
\renewcommand{\card}[1]{|#1|}

\newcommand{\crev}[1]{{\color{purple} #1}}

\newcommand{\abstracted}{We derive a sufficient condition for a $k$-th order homogeneous Markov chain $\Zvec$ with finite alphabet $\dom{Z}$ to have a unique invariant distribution on $\dom{Z}^k$. Specifically, \crev{let $\Xvec$ be a first-order, stationary Markov chain with finite alphabet $\dom{X}$ and a single recurrent class, let $g{:}\ \dom{X}\to\dom{Z}$ be non-injective, and define the (possibly non-Markovian) process $\Yvec:=g(\Xvec)$ (where $g$ is applied coordinate-wise). If $\Zvec$ is the $k$-th order Markov approximation of $\Yvec$, its invariant distribution is unique. We generalize this to non-Markovian processes $\Xvec$}.}

\begin{document}

\ifthenelse{\arxiv=1}{
\renewcommand{\crev}[1]{{\color{black} #1}}
\author{Bernhard C. Geiger}
\title{A Sufficient Condition for a Unique Invariant Distribution of a Higher-Order Markov Chain}
\maketitle

\begin{abstract}
  \abstracted
\end{abstract}
}
{
\begin{frontmatter}
  \title{A Sufficient Condition for a Unique Invariant Distribution of a Higher-Order Markov Chain}
  \author{Bernhard C. Geiger\corref{mycorrespondingauthor}}
  \address{Institute for Communications Engineering, Technical University of Munich,\\ Theresienstr. 90, 80333 Munich, Germany}
  \ead{bernhard.geiger@tum.de}
  \cortext[mycorrespondingauthor]{Corresponding author}

  \begin{abstract}
    \abstracted
  \end{abstract}
  


\end{frontmatter}
\linenumbers
}{}

\section{Introduction}
We consider invariant distributions of a $k$-th order (i.e., ``multiple'') Markov chain $\Zvec:=(Z_n)_{n\in\mathbb{N}_0}$ on a finite alphabet $\dom{Z}$. For $k=1$, the invariant distribution $\pi$ is a probability distribution on $\dom{Z}$ and is unique if the Markov chain has a single recurrent class~\cite[Thm.~4.4.2]{Gallager_StochasticProcesses}. If, in addition, the Markov chain is aperiodic, then the distribution of $Z_n$ converges to this invariant distribution as $n\to\infty$~\cite[Thm.~4.3.7]{Gallager_StochasticProcesses}. A fortiori, uniqueness and convergence are ensured if the Markov chain is \emph{regular}, i.e., irreducible and aperiodic~\cite[Thm.~4.1.6]{Kemeny_FMC}.

For $k$-th order Markov chains, $k>1$, two types of invariant distributions can be considered: A distribution $\pi$ on $\dom{Z}$ and a distribution $\mu$ on $\dom{Z}^k$. 

The invariant distribution $\pi$ on $\dom{Z}$ is related to the eigenvector problem of nonnegative tensors. Chang et al.\ showed that the eigenvector associated with the largest eigenvalue of an irreducible tensor is positive, but not necessarily simple~\cite[Thm.~1.4]{Chang_NonnegativeTensors}. The results were used by Li and Ng \crev{in~\cite{Li_Tensor}} to derive conditions under which there exists a unique distribution $\pi$ on $\dom{Z}$ such that, for a $k$-th order Markov chain $\Zvec$,
\begin{equation}
 \forall z\in\dom{Z}{:}\quad \pi_z = \sum_{z_0,\dots,z_{k-1}\in\dom{Z}} \Prob(Z_k=z|Z_{k-1}=z_{k-1},\dots,Z_0=z_0) \pi_{z_{k-1}}\cdots \pi_{z_0}.
\end{equation}
Regarding convergence, Vladimirescu showed that if $\Zvec$ is a regular second-order Markov chain, then there exists a distribution $\pi$ on $\dom{Z}$ such that~\cite[eq.~(7.1.3)]{Kalpazidou_CycleRepresentationsOfMarkovProcesses}
\begin{equation}
 \forall z,z_0,\dots,z_{k-1}\in\dom{Z}{:}\quad \pi_z=\limn \Prob(Z_n=z|Z_{k-1}=z_{k-1},\dots,Z_0=z_0).
\end{equation}

The more interesting case concerns invariant distributions $\mu$ on $\dom{Z}^k$. Following Doob~\cite[p.~89]{Doob_StochasticProcesses}, every $k$-th order Markov chain $\Zvec$ on $\dom{Z}$ can be converted to a first-order Markov chain $\Zvec^{(k)}:=(Z_n,\dots,Z_{n+k-1})_{n\in\mathbb{N}_0}$ on $\dom{Z}^k$. Kalpazidou used this fact to characterize invariant distributions for Markov chains $\Zvec$ derived from \emph{weighted circuits}~\cite[Prop.~7.2.2]{Kalpazidou_CycleRepresentationsOfMarkovProcesses}. If $\Zvec^{(k)}$ is a regular first-order Markov chain, then $\Zvec$ is a regular $k$-th order Markov chain~\cite[Prop.~7.1.6]{Kalpazidou_CycleRepresentationsOfMarkovProcesses}. Moreover, since in this case $\Zvec^{(k)}$ has a unique invariant distribution on $\dom{Z}^k$, so has $\Zvec$.

The converse, however, is not true: The regularity of $\Zvec^{(k)}$ does \emph{not} follow from the regularity of $\Zvec$, and hence the uniqueness of an invariant distribution $\mu$ on $\dom{Z}^k$ is not guaranteed even for a regular $k$-th order Markov chain $\Zvec$. \crev{To address this problem, Herkenrath discussed the uniform ergodicity of a second-order Markov process $\Zvec$ on a general alphabet $\dom{Z}$. Specifically, he defined the second-order Markov process $\Zvec$ to be uniformly ergodic if the first-order Markov process $\Zvec^{[2]}:=(Z_{2n},Z_{2n+1})_{n\in\mathbb{N}_0}$ on $\dom{Z}^2$ is uniformly ergodic~\cite[Def.~4]{Herkenrath_UniformErgodicity}. Herkenrath showed how sufficient and/or necessary conditions for uniform ergodicity (such as a strengthened Doeblin condition) carry over from $\Zvec$ to $\Zvec^{[2]}$~\cite[Lem.~3]{Herkenrath_UniformErgodicity}. He moreover presented sufficient conditions for uniform ergodicity of certain classes of second-order Markov processes, such as nonlinear autoregressive time series with absolutely continuous noise processes~\cite[Thms.~2-5]{Herkenrath_UniformErgodicity}. Herkenrath moreover showed a relation between the invariant distributions $\pi$ and $\mu$ on $\dom{Z}$ and $\dom{Z}^2$, respectively~\cite[Cor.~2]{Herkenrath_UniformErgodicity}:
\begin{equation}
 \forall z\in\dom{Z}{:}\quad \pi_z=\sum_{z'\in\dom{Z}} \mu_{z,z'} = \sum_{z'\in\dom{Z}} \mu_{z',z}.
\end{equation}}

In this work, we present a sufficient condition for \crev{a $k$-th order homogeneous Markov chain $\Zvec$ on a finite alphabet $\dom{Z}$} to have a unique invariant distribution $\mu$ on $\dom{Z}^k$. The condition is formulated via a function of a first-order Markov chain  $\Xvec$ with a single recurrent class. Since functions of Markov chains, so-called \emph{lumpings}, usually do not possess the Markov property, one may need to approximate this lumping by a Markov chain with a given order. Assuming that this Markov approximation satisfies certain conditions, it can be shown that its invariant distribution is unique. We moreover generalize this result by letting $\Xvec$ be a higher-order Markov chain and a non-Markovian process, respectively.

\section{Problem Setting}

We abbreviate vectors as $z_1^k:=(z_1,\dots,z_k)$ and random vectors as $Z_1^k:=(Z_1,\dots,Z_k)$. If the length of the vector is clear from the context, we omit indices. The probability that $Z=z$ is written as $\pmf{Z}(z):=\Prob(Z=z)$; the conditional probability that $Z_1=z_1$ given that $Z_2=z_2$ is written as $\pmf{Z_1|Z_2}(z_1|z_2)$. Stochastic processes are written as bold-faced letters, e.g., $\Zvec:=(Z_n)_{n\in\mathbb{N}_0}$. We write sets with calligraphic letters. For example, the alphabet of $\Zvec$ is $\dom{Z}$. All processes and random variables are assumed to live on a finite alphabet, i.e., $\card{\dom{Z}}<\infty$. The complement of a set $\dom{A}\subseteq\dom{Z}$ is $\dom{Z}^c:=\dom{Z}\setminus\dom{A}$. Transition probability matrices are written in bold-face, too; whether a symbol is a matrix or a stochastic process will always be clear from the context. We naturally extend a function $g{:}\ \dom{Z}\to\dom{W}$ from scalars to vectors by applying it coordinate-wise, i.e., $g(z_1^k):=(g(z_1),\dots,g(z_k))$. Similarly, the preimage of a vector is the Cartesian product of the preimages, i.e., $g^{-1}(w_1^k):=g^{-1}(w_1)\times\cdots\times g^{-1}(w_k)$.

 A stochastic process $\Zvec$ is a $k$-th order Markov chain with alphabet $\dom{Z}$ if and only if
 \begin{equation}\label{eq:markovK}
  \forall n\ge k{:}\quad \forall z_0^n\in\dom{Z}^{n+1}{:}\quad 
  \pmf{Z_n|Z_0^{n-1}}(z_n|z_0^{n-1}) = \pmf{Z_n|Z_{n-k}^{n-1}}(z_n|z_{n-k}^{n-1}).
 \end{equation}
If the right-hand side of~\eqref{eq:markovK} does not depend on $n$, we can write
 \begin{equation}\label{eq:hmarkovK}
  \Qment{z_{n-k}^{n-1}}{z_n} := \pmf{Z_n|Z_0^{n-1}}(z_n|z_0^{n-1})
 \end{equation}
 and call $\Zvec$ \emph{homogeneous}. We let $\Qm$ be a $\card{\dom{Z}}^k\times\card{\dom{Z}}$ matrix with entries $\Qment{z_{n-k}^{n-1}}{z_n}$ and abbreviate $\Zvec\sim\Mar{k,\dom{Z},\Qm}$. \crev{Similarly, we define
\begin{equation}
 \Qment{z_{0}^{k-1}}{z}^{(n)} := \pmf{Z_{k+n-1}|Z_0^{k-1}}(z|z_0^{k-1})
\end{equation}
 and collect the corresponding values in the $\card{\dom{Z}}^k\times\card{\dom{Z}}$ matrix $\Qm^{(n)}$. Note that $\Qm^{(1)}=\Qm$.
}

We recall basic definitions for $k$-th order Markov chains $\Zvec\sim\Mar{k,\dom{Z},\Qm}$; the details can be found in, e.g.,~\cite[Def.~7.1.1-7.1.4]{Kalpazidou_CycleRepresentationsOfMarkovProcesses} or~\cite[Def.~4.2.2-4.2.7]{Gallager_StochasticProcesses}. A state $z\in\dom{Z}$ is \emph{accessible} from $z'\in\dom{Z}$ (in short: $z' \to z$), if and only if
\begin{equation}
\forall u\in\dom{Z}^{k-1}{:}\quad \exists n=n(u,i,j){:}\quad \Qment{(u,z')}{z}^{(n)}>0.
\end{equation}
If $z' \to z$ and $z \to z'$, then $z$ and $z'$ \emph{communicate} (in short: $z\leftrightarrow z'$). A state $z$ is \emph{recurrent} if $z\to z'$ implies $z'\to z$, otherwise $z$ is \emph{transient}. The relation ``$\leftrightarrow$'' partitions the set $\{z\in\dom{Z}{:}\quad z\leftrightarrow z\}$ into equivalence classes (called recurrent classes). The Markov chain $\Zvec$ is \emph{irreducible} if and only if $\dom{Z}$ is the unique recurrent class; it is \emph{regular} if and only if there exists $n\ge 1$ such that $\Qm^{(n)}$ is a positive matrix. 

A $k$-th order Markov chain $\Zvec\sim\Mar{k,\dom{Z},\Qm}$ can be converted to a first-order Markov chain on $\dom{Z}^k$. Let $\Zvec^{(k)}:=(Z_n^{n+k-1})_{n\in\mathbb{N}_0}$. Then, $\Zvec^{(k)}\sim\Mar{1,\dom{Z}^k,\Pm}$, where~\cite[p.~89]{Doob_StochasticProcesses}
\begin{equation}
 \forall z_0^{k-1},{z'}_0^{k-1}\in\dom{Z}^k{:}\quad 
 \Pment{{z'}_0^{k-1}}{{z}_0^{k-1}} = \begin{cases}
                                      \Qment{{z'}_0^{k-1}}{z_{k-1}} & z'_1=z_0,\dots,z'_{k-1}=z_{k-2}\\
                                      0 & \text{else}.
                                     \end{cases}
\end{equation}

\begin{definition}[Invariant Distribution]\label{def:invariant}
 Let $\Zvec\sim\Mar{k,\dom{Z},\Qm}$. A distribution $\mu$ on $\dom{Z}^k$ is \emph{invariant} if and only if
\begin{equation}\label{eq:higher:invariance}
 \forall z_1^k \in \dom{Z}^{k}{:}\quad \mu_{z_1^k} = \sum_{z_0\in\dom{Z}} \mu_{{z}_0^{k-1}} \Qment{{z}_0^{k-1}}{z_k}.
\end{equation}
\end{definition}

It can be shown that a distribution $\mu$ on $\dom{Z}^k$ is invariant for $\Zvec$ if and only if it is invariant for $\Zvec^{(k)}$. If $\Zvec^{(k)}$ has a single recurrent class, then this $\mu$ is unique~\cite[Thm.~4.4.2]{Gallager_StochasticProcesses}. The following example illustrates that even if $\Zvec$ is regular, $\Zvec^{(k)}$ may have multiple recurrent classes:

\begin{example}
 Let $\Zvec\sim\Mar{2,\{1,2,3,4\},\Qm}$, where
 \begin{equation}
  \Qm = \bordermatrix{ & 1 & 2 & 3 & 4 \cr
      1,1 & 0.5 & 0.5 	& 0 	& 0 \cr
      1,2 & 0   & 0   	& 1 	& 0 \cr
      1,3 & 0   & 1   	& 0 	& 0 \cr
      1,4 & 0	& 0   	& 1 	& 0 \cr
      2,1 & 0   &     0 &   0.5 &    0.5 \cr
      2,2 & 0  	&  0.5 	&   0.5 &         0 \cr
      2,3 & 0.5 &   0 	&   0 	&   0.5 \cr
      2,4 & 1 	& 0   	&0 	&        0 \cr
      3,1 & 0	& 1	&  0	&   0 \cr
      3,2 & 1	&   0	&  0	&  0 \cr
      3,3 & 0   &  0.5  & 0.5	&    0 \cr
      3,4 & 1	&    0	&    0  &   0 \cr
      4,1 & 0   & 1   	& 0 	& 0 \cr
      4,2 & 0   & 1   	& 0 	& 0 \cr
      4,3 & 0   & 1   	& 0 	& 0 \cr
      4,4 & 0   &     0 &   0.5 &    0.5
      }.
 \end{equation}
 This Markov chain $\Zvec$ is regular since $\Qm^{(10)}>0$. $\Zvec$ is such that, depending on the initial states, we either observe sequences $1-2-3-4-1$ and $1-2-3-1$ or sequences $1-4-3-2-1$ and $1-3-2-1$. It follows that $\Zvec^{(2)}$ has transient states $\{(1,1),(2,2),(2,4),(3,3),(4,2),(4,4)\}$ and two recurrent classes:
 \[
  \{(1,2),(2,3),(3,1),(3,4),(4,1)\}
 \]
 and
 \[
  \{(1,3),(3,2),(2,1),(1,4),(4,3)\}.
 \]
 Hence, there is no unique invariant distribution $\mu$ satisfying~\eqref{eq:higher:invariance}.
\end{example}

We define the Markov approximation of a non-Markovian process:

\begin{definition}[$k$-th order Markov Approximation]\label{def:Mapprox}
 Let $\Wvec$ be a stationary stochastic process on the finite alphabet $\dom{Z}$. The \emph{$k$-th order Markov approximation} of $\Wvec$ is a $k$-th order Markov chain $\Zvec\sim\Mar{k,\dom{Z},\Qm}$ with transition matrix $\Qm$ with entries 
 \begin{equation}\label{eq:aggregation}
 \forall z_0^k\in\dom{Z}^{k+1}{:}\quad 
 \Qment{z_0^{k-1}}{z_k} = \begin{cases}
                           \frac{p_{W_0^{k}}(z_0^{k})}{p_{W_0^{k-1}}(z_0^{k-1})},& p_{W_0^{k-1}}(z_0^{k-1})>0\\
                           \frac{1}{\card{\dom{Z}}}, & \text{else.}
                          \end{cases}
\end{equation}
\end{definition}

The approximation~\eqref{eq:aggregation} can be justified via information theory. We define the relative entropy rate between $\Wvec$ and a $k$-th order Markov chain $\Mvec\sim\Mar{k,\dom{Z},\tilde\Qm}$ as~\cite[p.~266]{Gray_Entropy}
\begin{equation}\label{eq:kldr}
 \mathbb{D}(\Wvec\Vert\Mvec):= \limn \frac{1}{n} \sum_{z_0^{n-1}{:}\ p_{W_0^{n-1}}(z_0^{n-1})>0} p_{W_0^{n-1}}(z_0^{n-1}) \log \frac{p_{W_0^{n-1}}(z_0^{n-1})}{p_{M_0^{n-1}}(z_0^{n-1})}
\end{equation}
whenever the limit exists and is finite. Then, one can show that~\cite[Cor.~10.4]{Gray_Entropy}
\begin{equation}\label{eq:inf}
 \inf_{\tilde\Qm} \mathbb{D}(\Wvec\Vert\Mvec) = \mathbb{D}(\Wvec\Vert\Zvec)
\end{equation}
where $\Zvec\sim\Mar{k,\dom{Z},\Qm}$ with $\Qm$ defined in~\eqref{eq:aggregation}. Note that the minimizer of~\eqref{eq:inf} is not unique: Those sequences $z_0^{n-1}$ that contain a subsequence $z_\ell^{\ell+k-1}$ with $p_{W_0^{k-1}}(z_\ell^{\ell+k-1})=0$ are not part of the sum in~\eqref{eq:kldr}, hence the choice of $\Qment{z_\ell^{\ell+k-1}}{z}$ is immaterial. We will justify our particular choice later. The following example shows that approximating $\Wvec$ may lead to counterintuitive results if $\Wvec$ is a first-order Markov chain:

\begin{example}\label{ex:strange}
  Let $\Wvec\sim\Mar{1,\{1,2,3\},\Pm}$ with $\Pment{1}{3}=\Pment{2}{1}=\Pment{3}{1}=0$ and all other entries $\Pment{x_0}{x_1}>0$. $\Wvec$ has a single recurrent class $\{2,3\}$ and a transient state $\{1\}$, hence the unique invariant distribution $\pi$ satisfies $\pi_1=0$. Let $\Zvec\sim\Mar{1,\{1,2,3\},\Qm}$ be the first-order Markov approximation of $\Wvec$. Since $\pi_1=0$, we have that
  \begin{equation}
    \exists x\in\{1,2,3\}{:}\quad \Qment{1}{x}=\frac{1}{3}\neq\Pment{1}{x}
  \end{equation}
  and $\Zvec\not\equiv\Wvec$. Nevertheless, we get $\mathbb{D}(\Wvec\Vert\Zvec)=0$.

  Moreover, the $k$-th order Markov approximation of $\Wvec$ may even fail to be Markov of any order smaller than $k$. Indeed, we may have that
  \begin{equation}
  \exists x_1^{k}\in\{2,3\}^k{:}\quad \Qment{(1,x_1^{k-1})}{x_k}=\frac{1}{3}\neq\Qment{(2,x_1^{k-1})}{x_k}= \Pment{x_{k-1}}{x_k}.
  \end{equation}
\end{example}

\begin{lem}\label{lem:invariant}
 Let $\Wvec$ be a stationary stochastic process on the finite alphabet $\dom{Z}$ and let $\Zvec\sim\Mar{k,\dom{Z},\Qm}$ be its $k$-th order Markov approximation. Then, $p_{W_0^{k-1}}$ is an invariant distribution of $\Zvec$.
\end{lem}

\begin{proof}
If $p_{W_0^{k-1}}$ is invariant, then for every $z_1^k\in\dom{Z}^{k}$,
\begin{multline}
 p_{W_0^{k-1}}(z_1^k) = \sum_{z_0\in\dom{Z}} p_{W_0^{k-1}}(z_0^{k-1}) \Qment{z_0^{k-1}}{z_k}\\
 \stackrel{(a)}{=}\sum_{z_0\in\dom{Z}} p_{W_0^{k}}(z_0^{k}) = p_{W_1^{k}}(z_1^{k}) \stackrel{(b)}{=}p_{W_0^{k-1}}(z_1^{k})
\end{multline}
where $(a)$ is because those $z_0$ for which $p_{W_0^{k-1}}(z_0^{k-1})=0$ do not influence the sum, and where $(b)$ is because $\Wvec$ is stationary.
\end{proof}

\section{Main Result}
We present our main result:
\begin{thm}\label{thm:main}
Let $\Xvec\sim\Mar{1,\dom{X},\Pm}$ have a single recurrent class and invariant distribution $\pi$. Let $\Xvec$ be stationary, i.e., $\pmf{X_0}=\pi$. Let $g{:}\ \dom{X}\to\dom{Y}$ \crev{where $1<\card{\dom{Y}}\le\card{\dom{X}}$, i.e., $g$ may be non-injective}. Define $\Yvec$ via $Y_n=g(X_n)$, and let $\Zvec\sim\Mar{k,\dom{Y},\Qm}$ be the $k$-th order Markov approximation of $\Yvec$ with $\Qm$ given in~\eqref{eq:aggregation}. Then, $\Zvec$ has a unique invariant distribution $\mu$ on $\dom{Y}^k$ satisfying
\begin{equation}
 \forall y_0^{k-1}\in\dom{Y}{:}\quad \mu_{y_0^{k-1}} = \pmfy{_0^{k-1}} 
 = \sum_{x_0^{k-1}\in g^{-1}(y_0^{k-1})} \pi_{x_0} \prod_{\ell=1}^k \Pment{x_{\ell-1}}{x_\ell}.
\end{equation}
\end{thm}

\crev{
Theorem~\ref{thm:main} holds also for bijective functions. For these functions, however, $\Yvec$ is a first-order Markov chain and we may face counterintuitive results as illustrated in Example~\ref{ex:strange}. A similar situation may occur if $\Xvec$ is \emph{lumpable} w.r.t.\ the non-injective function $g$, i.e., if $\Yvec$ is a Markov chain~\cite[\S 6.3]{Kemeny_FMC}. Hence, Theorem~\ref{thm:main} is useful mainly in situations where $\Yvec$ is not Markov of any order.

Theorem~\ref{thm:main} holds for the definition of $\Qm$ in~\eqref{eq:aggregation} where the conditional distribution for a conditioning event with zero probability is chosen as the uniform distribution. More generally, it holds if the uniform distribution is replaced by any positive probability vector on $\dom{Z}$. This positivity constraint cannot be dropped, however, as the following example illustrates:
}

\begin{example}\label{ex:reducible}
 Let $\Xvec\sim\Mar{1,\{1,2,3\},\Pm}$ with $\Pment{1}{1}=0$ and all other entries $\Pment{x_0}{x_1}>0$. Let $g{:}\ \{1,2,3\}\to\{1,2\}$ be such that $g(1)=1$ and $g(2)=g(3)=2$. Suppose we want to model $\Yvec$ by a second-order Markov chain $\Zvec$. We get
  \begin{equation}
  \Qm = \bordermatrix{ & 1 & 2 \cr
      1,1 & ? & ? \cr
      1,2 & q & 1-q \cr
      2,1 & 0 & 1  \cr
      2,2 & p & 1-p 	
      }
 \end{equation}
 where $p,q\in(0,1)$ and where the $?$ indicates that the distribution of $\Yvec$ does not tell us how to choose these probabilities (the event $Y_0^1=(1,1)$ occurs with probability zero).
 
 Choosing $\Qment{(1,1)}{1}=1$ (and hence $\Qment{(1,1)}{2}=0$) results in $\Zvec^{(2)}$ having the two recurrent classes $\{(1,1)\}$ and $\{(1,2),(2,1),(2,2)\}$. The invariant distribution is thus not unique. Indeed, for every invariant distribution, the state $(1,1)$ has the same probability as it has in the initial distribution.  
\end{example}

The following result extends Theorem~\ref{thm:main} to the scenario where $\Yvec$ is the function of a $\ell$-th order Markov chain.

\begin{cor}\label{cor:higher}
Let $\Xvec\sim\Mar{\ell,\dom{X},\Pm}$ be such that $\Xvec^{(\ell)}$ has a single recurrent class and the unique invariant distribution $\pi$. Let $\Xvec^{(\ell)}$ be stationary, i.e., $\pmf{X_0^{\ell-1}}=\pi$. Let $g{:}\ \dom{X}\to\dom{Y}$. Define $\Yvec$ via $Y_n=g(X_n)$, and let $\Zvec\sim\Mar{k,\dom{Y},\Qm}$ be the $k$-th order Markov approximation of $\Yvec$ with $\Qm$ given in~\eqref{eq:aggregation}. Then, $\Zvec$ has a unique invariant distribution $\mu$ on $\dom{Y}^k$.
\end{cor}

\begin{proof}
 Note that $\Xvec^{(\ell)}$ is a first-order Markov chain on $\dom{X}^\ell$. By assumption, it has a single recurrent class. We define the function $g^{(\ell)}{:}\ \dom{X}^\ell \to \dom{Y}$ as
 \begin{equation}
  \forall x_0^{\ell-1}\in\dom{X}^\ell{:}\quad g^{(\ell)}(x_0^{\ell-1}) = g(x_0).
 \end{equation}
 Hence, $Y_n=g(X_n)=g^{(\ell)}(X_n^{n+\ell-1})$. Theorem~\ref{thm:main} completes the proof.
\end{proof}

\begin{figure}[t] 
\centering
  \begin{pspicture}[showgrid=false](-1,1)(6,6.25)
    \psset{style=Arrow}
    \pssignal(0,6){v}{$\Xvec$}
    \pssignal(4,6){m}{$\Mvec\sim\Mar{k,\dom{X},\Pm}$}
    \pssignal(0,3){y}{$\Yvec$}
    \pssignal(4,3){yt}{$\tilde\Yvec$}
    \pssignal(2,1){z}{$\Zvec\sim\Mar{k,\dom{Y},\Qm}$}
    \ncline{->}{v}{y} \aput{0}{$g$}
    \ncline{m}{yt}\aput{0}{$g$}
    \ncline[style=Dash]{v}{m} \aput{0}{$k$-App.}
    \ncline[style=Dash]{y}{z} \aput{0}{$k$-App.}
    \ncline[style=Dash]{yt}{z} \aput{0}{$k$-App.}
\end{pspicture}
\caption{The order of Markov approximation and projection through a non-injective function commutes. Dashed arrows labeled with ``$k$-App.'' denote the $k$-th order Markov approximation in the sense of Definition~\ref{def:Mapprox}, solid arrows labeled with ``$g$'' denote a projection of the stationary process through the non-injective function $g$. While the (generally non-Markovian) processes $\Yvec$ and $\tilde\Yvec$ differ, their $k$-th order Markov approximations are identical.}
\label{fig:commute}
\end{figure}
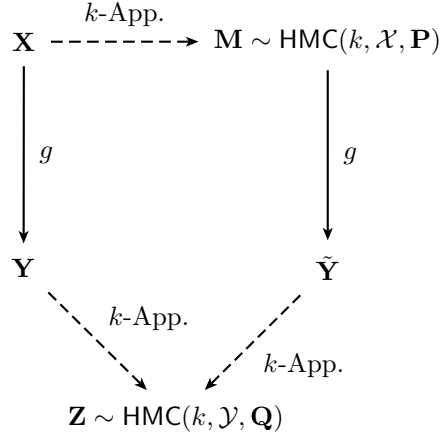

The condition in Corollary~\ref{cor:higher} holds for every pair of integers $\ell$ and $k$. However, requiring that $\Xvec$ has a unique invariant distribution on $\dom{X}^\ell$ is too restrictive if $k<\ell$. The next result shows that it suffices if the $k$-th order Markov approximation of $\Xvec$ has a unique invariant distribution on $\dom{X}^k$ to ensure that $\Zvec$ has a unique invariant distribution on $\dom{Y}^k$. We can thus drop the condition that $\Xvec$ is Markov by showing that the order of Markov approximation and projection through a non-injective function commutes in the sense of Figure~\ref{fig:commute}.

\begin{cor}\label{cor:stationary}
Let $\Xvec$ be a stationary stochastic process on the finite alphabet $\dom{X}$. Let $g{:}\ \dom{X}\to\dom{Y}$. Define $\Yvec$ via $Y_n=g(X_n)$, and let $\Zvec\sim\Mar{k,\dom{Y},\Qm}$ be the $k$-th order Markov model of $\Yvec$ with $\Qm$ given in~\eqref{eq:aggregation}. If the $k$-th order Markov approximation $\Mvec$ of $\Xvec$ is such that $\Mvec^{(k)}$ has a single recurrent class, then $\Zvec$ has a unique invariant distribution $\mu$ on $\dom{Y}^k$.
\end{cor}

\begin{proof}
 Let $\Mvec\sim\Mar{k,\dom{X},\Pm}$ be the $k$-th order Markov approximation of $\Xvec$. If $\Mvec^{(k)}$ has a single recurrent class, its invariant distribution is unique and equals, by Lemma~\ref{lem:invariant}, $\pi=p_{X_0^{k-1}}$. We set $p_{M_0^ {k-1}}=\pi$ such that $\Mvec$ is stationary and define $\tilde\Yvec$ via $\tilde{Y}_n=g(M_n)$. If we can show that the $k$-th order Markov approximation of $\tilde\Yvec$ coincides with the $k$-th order Markov approximation of $\Yvec$, then the proof follows from Corollary~\ref{cor:higher}. It therefore suffices to show that $p_{Y_0^{k}}=p_{\tilde{Y}_0^{k}}$. We have
 \begin{align}
  \forall y_0^k\in\dom{Y}^{k+1}{:}\quad p_{\tilde{Y}_0^{k-1}}(y_0^k) &=
  \sum_{x_0^k\in g^{-1}(y_0^k)} p_{M_0^k}(x_0^k)\\
  &\stackrel{(a)}{=}\sum_{x_0^k\in g^{-1}(y_0^k)} p_{X_0^{k-1}}(x_0^{k-1}) \Pment{x_0^{k-1}}{x_k}\\
  &\stackrel{(b)}{=}\sum_{x_0^k\in g^{-1}(y_0^k)} p_{X_0^{k-1}}(x_0^{k-1}) \frac{p_{X_0^{k}}(x_0^{k})}{p_{X_0^{k-1}}(x_0^{k-1})}\\
  &= \sum_{x_0^k\in g^{-1}(y_0^k)} p_{X_0^{k}}(x_0^{k})=p_{Y_0^k}(y_0^k)
 \end{align}
where in $(a)$ we used the fact that $p_{X_0^{k-1}}$ is the unique invariant (thus stationary) distribution of $\Mvec$ and where $(b)$ follows because those $x_0^{k}\in g^{-1}(y_0^k)$ for which $p_{X_0^{k-1}}(x_0^{k-1})=0$ do not influence the sum. This completes the proof.
\end{proof}

\section{Proof of Theorem~\ref{thm:main}}
That $\mu=p_{Y_0^{k-1}}$ is an invariant distribution for $\Zvec$ follows from Lemma~\ref{lem:invariant}. To show that this invariant distribution is unique, we show that the first-order Markov chain $\Zvec^{(k)}$ has the single recurrent class
\begin{equation}
 \dom{S} := \{y_0^{k-1}\in\dom{Y}{:}\quad \pmfy{_0^{k-1}}>0\}.
\end{equation}
We show in Section~\ref{ssec:recurrent} that all states in $\dom{S}$ communicate by showing that for all $y,y'\in\dom{S}$ we have $y\to y'$ (hence, $y\leftrightarrow y'$). We show in Section~\ref{ssec:class} that $\dom{S}$ is a recurrent class by showing that, for $y\in\dom{S}$ and $y'\in\dom{S}^c$, we have $y\not\to y'$. Finally, in Section~\ref{ssec:transient} we show that the states in $\dom{S}^c$ are transient by showing that for every $y'\in\dom{S}^c$ there is an $y\in\dom{S}$ such that we have $y'\to y$. This completes the proof that $\dom{S}$ is the single recurrent class of $\Zvec^{(k)}$, from which the uniqueness of $\mu$ follows~\cite[Thm.~4.4.2]{Gallager_StochasticProcesses}.

\subsection{All states in $\dom{S}$ communicate}\label{ssec:recurrent}
\begin{lem}\label{lem:absolutecontinuity}
 For every $n\ge 0$ and every $y_0^{k-1}\in\dom{Y}^k$ such that $\pmfy{_0^{k-1}}>0$, we have
 \begin{equation}\label{eq:abscont}
  \forall y_{k}^{k+n}\in\dom{Y}^{n+1}{:}\quad \pmfm[_0^{k-1}]{_{k}^{k+n}}=0 \Rightarrow \pmfy[_0^{k-1}]{_{k}^{k+n}}=0.
 \end{equation}
\end{lem}

\begin{proof}
 This result is an immediate consequence of~\cite[Thm.~10.1,~eq.~(10.8)]{Gray_Entropy}. There, it was shown that with~\eqref{eq:aggregation} and with setting $\pmfm{_0^{k-1}}=\pmfy{_0^{k-1}}$, one gets for every $n\ge 0$
 \begin{equation}
  \forall y_{0}^{k+n}\in\dom{Y}^{k+n+1}{:}\quad \pmfm{_{0}^{k+n}}=0 \Rightarrow \pmfy{_{0}^{k+n}}=0.
 \end{equation}
 The proof follows by dividing both sides by $\pmfy{_0^{k-1}}>0$.
\end{proof}

We show that all states in $\dom{S}$ communicate by showing that, for every pair $y,y'\in\dom{S}$, we have $y\to y'$. With Lemma~\ref{lem:absolutecontinuity} it thus suffices to show that for every pair $y,y'\in\dom{S}$ there exists an $n=n(y,y')> 0$ such that
\begin{equation}\label{eq:recurrent:condition}
 \pmf{Y_{n}^{n+k-1}|Y_0^{k-1}}(y'|y)>0.
\end{equation}
We can write this as
\begin{align}
& \pmf{Y_{n}^{n+k-1}|Y_0^{k-1}}(y'|y) \notag\\
&= \sum\limits_{\substack{x_0^{k-1}\in g^{-1}(y)\\{x'}_n^{n+k-1}\in g^{-1}(y')}} 
    \pmf{X_n^{n+k-1}|X_0^{k-1}}({x'}_n^{n+k-1}|x_0^{k-1}) \pmf{X_0^{k-1}|Y_0^{k-1}}(x_0^{k-1}|y)\notag\\
&= \sum\limits_{\substack{x_0^{k-1}\in g^{-1}(y)\\{x'}_n^{n+k-1}\in g^{-1}(y')}} 
    \pmf{X_{\max\{k,n\}}^{n+k-1}|X_{k-1}}({x'}_{\max\{k,n\}}^{n+k-1}|x_{k-1}) \pmf{X_0^{k-1}|Y_0^{k-1}}(x_0^{k-1}|y).\label{eq:recurrent:lowerbound}
\end{align}
Since $y\in\dom{S}$ and $y'\in\dom{S}$, there exist $x_0^{k-1}\in g^{-1}(y)$ and ${x'}_n^{n+k-1}\in g^{-1}(y')$ such that $\pmf{X_0^{k-1}}(x_0^{k-1})>0$ and $\pmf{X_0^{k-1}}({x'}_n^{n+k-1})>0$. It follows that $x_{k-1}$ and $x'_{\max\{k,n\}}$ are elements of the single recurrent class of $\Xvec$. Moreover, it follows that $\pmf{X_0^{k-1}|Y_0^{k-1}}(x_0^{k-1}|y)>0$ and that
\begin{equation}\label{eq:recurrent:product}
 \pmf{X_0^{k-1}}({x'}_n^{n+k-1}) = \pi_{x'_n} \prod_{\ell=n+1}^{n+k-1} \Pment{x'_{\ell-1}}{x'_\ell} >0.
\end{equation}
Since $\Xvec$ is Markov, the first term in the sum in~\eqref{eq:recurrent:lowerbound} can be written as
\begin{multline}
 \pmf{X_{\max\{k,n\}}^{n+k-1}|X_{k-1}}({x'}_{\max\{k,n\}}^{n+k-1}|x_{k-1}) \\
 = \underbrace{\pmf{X_{\max\{1,n-k+1\}}|X_0}(x'_{\max\{k,n\}}|x_{k-1})}_%
 {=\Pment{x_{k-1}}{x'_{\max\{k,n\}}}^{(\max\{1,n-k+1\})}}
 \cdot \prod_{\ell=\max\{k,n\}+1}^{n+k-1} \Pment{x'_{\ell-1}}{x'_\ell}.
\end{multline}
Since $x_{k-1}$ and $x'_{\max\{k,n\}}$ are in the same recurrent class, it follows that there exists an $n=n(x_{k-1},x'_{\max\{k,n\}})$ such that the first term in the product is positive; the second term is positive by~\eqref{eq:recurrent:product}. Combined with $\pmf{X_0^{k-1}|Y_0^{k-1}}(x_0^{k-1}|y)>0$ it follows that at least one summand in~\eqref{eq:recurrent:lowerbound} is positive, from which~\eqref{eq:recurrent:condition} follows.

\subsection{$\dom{S}$ is a recurrent class}\label{ssec:class}
We already showed that, for all $y,y'\in\dom{S}$, $y\leftrightarrow y'$. To show that $\dom{S}$ is a recurrent class, i.e., an equivalence class under the relation ``$\leftrightarrow$'', we must show that for $y\in\dom{S}$ and $y'\in\dom{S}^c$, we have $y\not\to y'$, i.e.,
\begin{equation}\label{eq:class:condition}
 \forall y\in\dom{S}, {y'}\in\dom{S}^c{:}\ \forall n\ge 1{:}\quad 
 \pmf{Z_n^{k+n-1}|Z_0^{k-1}}({y'}|y)=0.
\end{equation}
Suppose the contrary is true, i.e., there exists a $y_0^{k-1}\in\dom{S}$ and an $n\ge 1$, such that for a $y_n^{n+k-1}\in\dom{S}^c$ we have $\pmfm[_0^{k-1}]{_n^{n+k-1}}>0$. From this follows that there exists at least one sequence $y_k^{n+k-1}\in\dom{Y}^{n}$ such that
\begin{equation}\label{eq:class:positive}
 \pmfm[_0^{k-1}]{_k^{n+k-1}}=\prod_{\ell=k}^{n+k-1} \Qment{y_{\ell-k}^{\ell-1}}{y_\ell} >0.
\end{equation}
Since, by assumption, $y_0^{k-1}\in\dom{S}$ and $y_n^{n+k-1}\in\dom{S}^c$, there must be an $\ell\in\{k,\dots,k+n-1\}$ such that $y_{\ell-k}^{\ell-1}\in\dom{S}$ and $y_{\ell-k+1}^\ell \in\dom{S}^c$. But
\begin{align*}
 \pmf{Y_0^{k-1}}(y_{\ell-k+1}^\ell)&\stackrel{(a)}{=} \pmf{Y_1^{k}}(y_{\ell-k+1}^\ell)\\
 &= \sum_{y\in\dom{Y}} \pmf{Y_0^{k}}(y,y_{\ell-k+1}^\ell)\\
 &\ge \pmf{Y_0^{k}}(y_{\ell-k}^\ell)\\
 &= \pmf{Y_k|Y_0^{k-1}}(y_\ell|y_{\ell-k}^{\ell-1}) \pmf{Y_0^{k-1}}(y_{\ell-k}^{\ell-1})\\
 &\stackrel{(b)}{=} \Qment{y_{\ell-k}^{\ell-1}}{y_\ell}  \pmf{Y_0^{k-1}}(y_{\ell-k}^{\ell-1})\\
 &\stackrel{(c)}{>} 0
\end{align*}
where $(a)$ is due to stationarity of $\Yvec$ and where $(b)$ is because $y_{\ell-k}^{\ell-1}\in\dom{S}$ and because of~\eqref{eq:aggregation} in Definition~\ref{def:Mapprox}. Finally, $(c)$ follows from $y_{\ell-k}^{\ell-1}\in\dom{S}$ and~\eqref{eq:class:positive}. Hence $y_{\ell-k+1}^\ell \in\dom{S}$, a contradiction.

\subsection{All states in $\dom{S}^c$ are transient}\label{ssec:transient}
To show that all states in $\dom{S}^c$ are transient, we must show that for every $y\in\dom{S}^c$ there is at least one $y'\in\dom{S}$ such that $y\to y'$. Since we showed in the last section that $y'\not\to y$, it follows that $y$ is transient.

Suppose $y_0^{k-1}\in\dom{S}^c$. From~\eqref{eq:aggregation} we get
\begin{equation}
 \forall y_k\in\dom{Y}{:}\quad \pmfm[_0^{k-1}]{_k}=\frac{1}{\card{\dom{Y}}}.
\end{equation}
If for at least one $y_k$ we have $y_1^k\in\dom{S}$, then we are done. If for every $y_k$ we have $y_1^k\in\dom{S}^c$, then
\begin{equation}
 \forall y_k^{k+1}\in\dom{Y}^2{:}\quad \pmfm[_0^{k-1}]{_k^{k+1}}=\frac{1}{\card{\dom{Y}}^2}.
\end{equation}
If for at least one $y_k^{k+1}$ we have $y_2^{k+1}\in\dom{S}$, then we are done. Otherwise, we repeat above procedure. Eventually, if for every $y_k^{2k-2}\in\dom{Y}^{k-1}$ we have $y_{k-1}^{2k-2}\in\dom{S}^c$, then
\begin{equation}
 \forall y_k^{2k-1}\in\dom{Y}^k{:}\quad \pmfm[_0^{k-1}]{_k^{2k-1}}=\frac{1}{\card{\dom{Y}}^k}.
\end{equation}
But at least one $y_k^{2k-1}\in\dom{Y}^k$ must be such that $y_k^{2k-1}\in\dom{S}$, since $\dom{S}$ is non-empty. Thus, from every $y\in\dom{S}^c$ there is at least one $y'\in\dom{S}$ such that $y\to y'$.

\crev{
\section{Conclusion and Outlook}
A $k$-th order homogeneous Markov chain $\Zvec$ with finite alphabet $\dom{Z}$ has a unique invariant distribution on $\dom{Z}^k$ if the first-order Markov chain $\Zvec^{(k)}$ has a single recurrent class. We presented a sufficient condition for this to be the case: $\Zvec$ has a unique invariant distribution on $\dom{Z}^k$ if it is the $k$-th order Markov approximation of a function of a first-order Markov chain $\Xvec$ with a single recurrent class. This condition has practical relevance in, e.g., state space reduction for Markov chains, e.g.,~\cite{GeigerWu_HigherOrder,Meyn_MarkovAggregation}. We generalized our result to $\Xvec$ being a Markov chain of any order and to $\Xvec$ being not Markov at all.

Example~\ref{ex:strange} suggests that our $k$-th order Markov approximation in Definition~\ref{def:Mapprox} leads to counterintuitive results if the process to be approximated is a Markov chain of order smaller than or equal to $k$. Future work shall investigate whether a different choice of $\Qm$ in Definition~\ref{def:Mapprox} can alleviate this problem while still ensuring that Theorem~\ref{thm:main} holds.
}
 
\section*{References}
\bibliographystyle{elsarticle-num}
\bibliography{references}

\section*{Funding}
The work was funded by the Erwin Schr\"odinger Fellowship J 3765 of the Austrian Science Fund.
\end{document}